\documentclass{amsart}
\usepackage{amssymb}
\usepackage{amsfonts}
\usepackage{dsfont}
\usepackage{hyperref}

\newtheorem{theorem}{Theorem}[section]
\newtheorem{lemma}[theorem]{Lemma}
\newtheorem{corollary}[theorem]{Corollary}

\theoremstyle{definition}
\newtheorem{definition}[theorem]{Definition}

\newtheorem{problem}[theorem]{Problem}
\newtheorem{conjecture}[theorem]{Conjecture}
\theoremstyle{remark}
\newtheorem{remark}[theorem]{Remark}

\numberwithin{equation}{section}

\newcommand{\PP}{\mathbb P}
\newcommand{\EE}{\mathbb E}
\newcommand{\cA}{\mathcal A}
\newcommand{\cB}{\mathcal B}
\newcommand{\cF}{\mathcal F}
\newcommand{\cX}{\mathcal X}
\newcommand{\Far}{\operatorname{Far}}
\newcommand{\Cov}{\operatorname{Cov}}
\newcommand{\one}{\mathbf 1}
\newcommand{\ind}{\mathds 1}

\begin{document}

\title{On $p$-spread measures}

\author{Chen Li}
\address{Department of Mathematics, Courant Institute School of Mathematics, Computing, and Data Science, New York University, New York, New York 10012}
\curraddr{}
\email{chen.li1@cims.nyu.edu}
\thanks{}


\date{}

\dedicatory{}


\begin{abstract}
We study $p$-spread probability measures on the Boolean lattice.
We show that if a family of sets $\mathcal{A}$ is large under the product Bernoulli-$p$ measure, then no $p$-spread measure can be supported on sets that are not covered by the union of two members of $\mathcal{A}$, answering the fractional version of Talagrand's discrete convexity problem.
Consequently, we establish a coupling theorem between $p$-spread and product Bernoulli-$p$ measures.
\end{abstract}

\maketitle

\section{Introduction}\label{sec:intro}
\subsection{Background and main results}
Talagrand's convexity problem asks if one can create convexity by applying a bounded number of Minkowski sums to a large but otherwise arbitrary set in high-dimensional Gaussian space \cite{Talagrand1995, Talagrand2010, Talagrand2026}.
In \cite{HST}, Hua, Song, and Tudose resolved this problem, building on Song's formulation in \cite{Song2026} that the convexity problem is equivalent to a Gaussian-sum problem for sub-Gaussian vectors.

Talagrand also raised two discrete counterparts of the convexity problem: one integral and one fractional, replacing Gaussian variables with Bernoulli random variables and Minkowski sums with unions \cite{Talagrand1995, Talagrand2010, Talagrand2021, Talagrand2026}.
This paper settles the fractional version of the discrete convexity problem.

We now introduce the two problems.
For a positive integer $N$, write $[N]=\{1,\ldots,N\}$ and identify
$2^{[N]}$ with $\{0,1\}^N$.
For a probability measure $\eta$ on $2^{[N]}$ and $S\subseteq[N]$, we abbreviate $\eta(\{S\})$ as $\eta(S)$.
For $p\in(0,1)$, let $\mu_p$ be the product Bernoulli-$p$ measure on $2^{[N]}$:
\[
    \mu_p(S)=p^{|S|}(1-p)^{N-|S|}
    \qquad\text{for every }S\subseteq[N].
\]
For $I\subseteq[N]$, define its \emph{principal up-class} as
\[
    H_I=\{S\subseteq[N]:I\subseteq S\}.
\]
For a positive integer $m$ and a family of sets $\cA\subseteq2^{[N]}$, define
\[
    \Cov_m(\cA)
    =\bigl\{X\subseteq[N]:
    X\subseteq A_1\cup\cdots\cup A_m
    \text{ for some }A_1,\ldots,A_m\in\cA\bigr\},
\]
and
\[
    \Far_m(\cA)=2^{[N]}\setminus\Cov_m(\cA).
\]
When $\cA$ is large, $\Far_m(\cA)$ is, in some sense, very small.
What can we say about its structural properties?
The structure of interest for the discrete convexity problem is \emph{$p$-smallness}.
\begin{definition}\label{def:p-small}
A family $\cX\subseteq2^{[N]}$ is \emph{$p$-small} if there is a collection $\mathcal I\subseteq2^{[N]}$ such that
\[
    \cX\subseteq\bigcup_{I\in\mathcal I}H_I
    \qquad\text{and}\qquad
    \sum_{I\in\mathcal I}p^{|I|}\le\frac12.
\]
\end{definition}
\begin{remark}
Simply put, $\mathcal I$ contains explicit witnesses of $\cX$ being small.
This definition was motivated by the fact that the complements of universal dilations of large balanced convex sets admit half-space covers with a small total Gaussian measure \cite[Section~2.12]{Talagrand2021}, a consequence of the majorizing measure theorem.
\end{remark}
The discrete convexity problem (\cite[Problem~4.3]{Talagrand1995}, \cite[Conjecture~7.1]{Talagrand2010}, \cite[Research~Problem~13.3.2]{Talagrand2021}, \cite[Problem~3.3]{Talagrand2026}) can be stated as follows:
\begin{problem}[Discrete convexity]\label{prob:discrete-convexity}
Does there exist a positive integer $m$ such that, for every $N\ge1$, $p\in(0,1)$ and $\cA\subseteq2^{[N]}$ with $\mu_p(\cA)\ge3/4$, the family $\Far_m(\cA)$ is $p$-small?
\end{problem}
\begin{remark}
The discrete convexity problem is related to the study of \emph{thresholds}.
Problem~\ref{prob:discrete-convexity} has a reformulation in the language of \emph{$k$-thresholds} by Ascoli, He, Park, and Talagrand, who also proved its validity for certain classes of graph containment properties in \cite{AHPT}.
In addition, as a consequence of a theorem by Park and Pham \cite[Theorem~1.1]{PP2024} resolving the expectation threshold conjecture of Kahn and Kalai \cite[Conjecture~1.1]{KK2007}, we can conclude the following: There exists $C>0$ such that, under the assumption of Problem~\ref{prob:discrete-convexity}, $\Far_1(\cA)$ is $\bigl(p/[C\log(\max\{N,2\})]\bigr)$-small.
Moreover, the logarithmic factor is necessary, so Problem~\ref{prob:discrete-convexity} cannot be solved positively by taking $m=1$.
\end{remark}
Problem~\ref{prob:discrete-convexity} is open. 
In \cite[Corollary~1.4]{HST}, the following result is derived as a consequence of the resolution of the convexity problem and the argument in \cite[Proof~of~(3.7)]{Talagrand1995}:
\begin{theorem}\label{thm:HST}
There exist $\varepsilon>0$ and $L>1$ such that for every $N\ge1$, $p\in(0,1)$, and $\cA\subseteq2^{[N]}$ with $\mu_p(\cA)\geq 1-\varepsilon$, the family $\Far_3(\cA)$ is $p^L$-small.
\end{theorem}
\begin{remark}
In Theorem~\ref{thm:HST}, we can take $\varepsilon = 1/4$ at the expense of increasing $L$:
For independent $R_1,\ldots,R_k$ with law $\mu_p$, $R=\bigcap_{j=1}^kR_j$ has law $\mu_{p^k}$, and $R\in\Cov_1(\cA)$ if $R_j\in\Cov_1(\cA)$ for any $j$.
Therefore, $\mu_p(\cA)\ge3/4$ implies
\[
 \mu_{p^k}\bigl(\Far_1(\cA)\bigr)\le\bigl(1-\mu_p(\cA)\bigr)^k\le\frac{1}{4^k}.
\]
As a result, Problem~\ref{prob:discrete-convexity} holds for $m=3$ under a weaker conclusion of $p^L$-smallness.
\end{remark}

Since it is not always the easiest to find the witness class $\mathcal I$ when working with Definition~\ref{def:p-small}, Talagrand developed in \cite{Talagrand2010} the notion of \emph{weakly $p$-small} sets, a relaxation of $p$-smallness that allows witnesses to carry fractional weights.
\begin{definition}\label{def:weakly-p-small}
A family $\cX\subseteq2^{[N]}$ is \emph{weakly $p$-small} if there are weights $w_I\ge0$, indexed by $I\subseteq[N]$, such that
\[
    \sum_{I\subseteq S}w_I\ge1\quad\text{for every }S\in\cX,
    \qquad
    \sum_{I\subseteq[N]}w_Ip^{|I|}\le\frac12.
\]
\end{definition}
The fractional version of the discrete convexity problem \cite[Conjecture~7.2]{Talagrand2010} replaces ``$p$-small'' with the weaker statement ``weakly $p$-small'' in Problem~\ref{prob:discrete-convexity}.
It is unclear whether this modification results in a different problem.
Indeed, Talagrand conjectured (\cite[Conjecture~6.3]{Talagrand2010}, \cite[Conjecture~3.9]{Talagrand2026}) that Definition~\ref{def:p-small} and Definition~\ref{def:weakly-p-small} are equivalent up to a constant:
\begin{conjecture}\label{conj:p-small-equi}
    There exists $\alpha>0$ such that for every $N\ge1$, $p\in(0,1)$, and $\cX\subseteq2^{[N]}$, if $\cX$ is weakly $p$-small, then $\cX$ is $(\alpha p)$-small.
\end{conjecture}
The following theorem of Park \cite[(3)]{Park2026} shows that Conjecture~\ref{conj:p-small-equi} holds modulo a $\log\log(1/p)$ factor.
\begin{theorem}\label{thm:p-small-loglog}
    There exists $K>0$ such that for every $N\ge1$, $p\in(0,1)$, and $\cX\subseteq2^{[N]}$, if $\cX$ is weakly $p$-small, then $\cX$ is $\bigl(p/[K\max\{1,\log\log(1/p)\}]\bigr)$-small.
\end{theorem}
An advantage of working with weakly $p$-small sets is that we can exploit the dual concept of \emph{$p$-spreadness}.
\begin{definition}
A probability measure $\nu$ on $2^{[N]}$ is called \emph{$p$-spread} if
\begin{equation}\label{eq:spread}
    \nu(H_I)\le p^{|I|}\qquad\text{for every }I\subseteq[N].
\end{equation}
\end{definition}
\begin{remark}\label{rem:duality}
For \(\cX\subseteq2^{[N]}\) and a probability measure \(\nu\) on
\(2^{[N]}\), we say that \(\cX\) \emph{carries} \(\nu\) if \(\nu(\cX)=1\).
Observe that a $p$-small family cannot carry a $p$-spread measure.
Furthermore, \cite[Proposition~6.7]{Talagrand2010} relates weak smallness to the absence of spread measures, with a factor of two change in the parameter: a weakly $p$-small class of sets does not carry a $p$-spread measure, and any class of sets that does not carry a $(2p)$-spread measure is weakly $p$-small.
\end{remark}

The duality between spread measures and weak smallness was the starting point of Frankston, Kahn, Narayanan, and Park's proof \cite[Theorem~1.1]{FKNP2021} of Talagrand's fractional version of the expectation threshold conjecture \cite[Conjecture~8.2]{Talagrand2010}.
Similarly, we begin our investigation by phrasing Talagrand's relaxation of the discrete convexity problem in the language of $p$-spread measures:
\begin{problem}[Fractional discrete convexity]\label{prob:frac-discrete-convexity}
Does there exist a positive integer $m$ such that, for every $N\ge1$,
$p\in(0,1)$ and $\cA\subseteq2^{[N]}$ with $\mu_p(\cA)\ge3/4$, the family
$\Far_m(\cA)$ does not carry a $p$-spread probability measure?
\end{problem}
\begin{remark}
As noted in \cite[Section 6]{Talagrand2010}, $\Far_1(\cA)$ can be arbitrarily small in $\mu_p$ while still carrying a $p$-spread measure.
Therefore, as with Problem~\ref{prob:discrete-convexity}, Problem~\ref{prob:frac-discrete-convexity} cannot be solved positively by taking $m=1$.
\end{remark}

Our main result answers Problem~\ref{prob:frac-discrete-convexity} positively in a strong form: $m=2$ suffices (thus also answering positively \cite[Problem~3.10]{Talagrand2026}); moreover, our theorem provides a quantitative estimate of $\mu_p(\cA)$ when $\Far_2(\cA)$ carries a $p$-spread measure.
\begin{theorem}\label{thm:quantitative}
Let $N\ge1$, $p\in(0,1)$, and $\cA\subseteq2^{[N]}$.  
If $\Far_2(\cA)$ carries a $p$-spread probability measure, then
\[
  \mu_p(\cA)\le \frac{1-p}{2-p}.
\]
\end{theorem}
By tracking an additional factor of $2$ in the proof, Theorem~\ref{thm:quantitative} also establishes the following:
\begin{theorem}\label{thm:quantitative2}
Let $N\ge1$, $p\in(0,1)$, and $\cA\subseteq2^{[N]}$.  
If
\[
  \mu_p(\cA)\ge\frac{2-2p}{3-2p},
\]
then $\Far_2(\cA)$ is weakly $p$-small.
\end{theorem}

We prove Theorem~\ref{thm:quantitative} and Theorem~\ref{thm:quantitative2} in Section~\ref{sec:proof}.
Because $(1-p)/(2-p)<1/2$ for every $p\in(0,1)$, Theorem~\ref{thm:quantitative} answers Problem~\ref{prob:frac-discrete-convexity} affirmatively with $m=2$:
\begin{corollary}[Fractional discrete convexity conjecture]\label{cor:frac}
For every $N\ge1$, $p\in(0,1)$, and $\cA\subseteq2^{[N]}$ with $\mu_p(\cA)>1/2$, the family $\Far_2(\cA)$ carries no $p$-spread measure.
\end{corollary}
\begin{remark}
    In the parallel story of the convexity problem in Gaussian space, \cite[Proposition~2.6]{Talagrand1995} shows that $m=2$ (two-fold Minkowski sum) does not suffice for a positive solution.
\end{remark}
As observed in \cite[Corollary~3.5]{Park2026}, combining Corollary~\ref{cor:frac} with Theorem~\ref{thm:p-small-loglog} and Remark~\ref{rem:duality} results in the following partial progress towards Problem~\ref{prob:discrete-convexity}.
\begin{corollary}
There exists $K>0$ such that for every $N\ge1$, $p\in(0,1)$, and $\cA\subseteq2^{[N]}$ with $\mu_p(\cA)>1/2$, the family $\Far_2(\cA)$ is $\bigl(p/[K\max\{1,\log\log(1/p)\}]\bigr)$-small.
\end{corollary}

Ultimately, we hope this paper sheds some light on the mysterious nature of $p$-spread measures, a sentiment shared by Talagrand in \cite[Section~6]{Talagrand2010} and \cite[Section~3]{Talagrand2026}.
In Section~\ref{sec:comparison}, we show the equivalence between Corollary~\ref{cor:frac} and the following coupling result between $p$-spread and product Bernoulli-$p$ measures, which may be of independent interest.
\begin{theorem}[$p$-spread coupling theorem]\label{thm:coup}
Let $p\in(0,1)$ and $\nu$ be a $p$-spread measure on $2^{[N]}$.
There are jointly distributed random sets $X,R_1,R_2\in2^{[N]}$ with marginal laws $\nu,\mu_p,\mu_p$, respectively, such that
\[
    X\subseteq R_1\cup R_2.
\]
\end{theorem}
Theorem~\ref{thm:coup} was motivated by finding an equivalent formulation of Problem~\ref{prob:frac-discrete-convexity} that is analogous to the reformulation of the Gaussian convexity problem in \cite{Song2026}.
Indeed, our proof shares similar ingredients with \cite{Song2026} and \cite{HST}.

\subsection{Proof strategy}
According to Talagrand, a central difficulty underlying the discrete convexity problems is that we do not have a good handle on the operation ``taking the union of two sets in $\cA$'', and the main challenge is to find a way to express that the resulting set family is, in some sense, ``much larger than $\cA$'' \cite[Section~3]{Talagrand2026}.
Our approach to Theorems~\ref{thm:quantitative}~and~\ref{thm:quantitative2} borrows techniques from the study of intersecting families to work with the union operation by taking advantage of its binary nature.
In particular, we adopt the spectral methods developed by Friedgut \cite{Friedgut2008} and by Ellis, Filmus, and Friedgut \cite{EFF2012} to study the maximum size of \emph{$t$-intersecting} families and \emph{triangle-intersecting} families, respectively.
We next illustrate their relevance to Problems~\ref{prob:discrete-convexity}~and~\ref{prob:frac-discrete-convexity} by considering a particular example from the latter result.

Suppose $N=\binom{n}{2}$ for a positive integer $n$, and $2^{[N]}$ represents the set of subgraphs of the complete graph $K_n$ on $n$ vertices.
A family $\cF \subseteq 2^{[N]}$ is \emph{triangle intersecting} if, for every $G,H\in\cF$, $G\cap H$ contains a triangle.
The following conjecture of Simonovits and S\'os was settled in \cite{EFF2012}.
\begin{theorem}[Simonovits-S\'os conjecture]\label{thm:simonovits-sos}
    For any triangle-intersecting family $\cF$, $\mu_{1/2}(\cF)\le\frac18$.
\end{theorem}
In other words, under the uniform measure, any triangle-intersecting family cannot be larger than the principal up-class generated by a particular triangle.

To motivate our proof of Theorems~\ref{thm:quantitative}~and~\ref{thm:quantitative2}, let us emphasize a couple of key steps shared with the proof of Theorem~\ref{thm:simonovits-sos}.
See \cite[Section~1.3]{EFF2012} for a narrative of the proof methods for the Simonovits-S\'os conjecture.
For a set $S\subseteq[N]$, we use $S^c=[N]\setminus S$ to denote the complement of $S$ and $S\mathbin{\triangle} S'$ to denote the symmetric difference of $S,S'\subseteq[N]$.

As stated in \cite[Section~1.2]{EFF2012}, the starting point of the proof of Theorem~\ref{thm:simonovits-sos} is the following observation, also used by Chung, Graham, Frankl, and Shearer to obtain a bound of $\mu_{1/2}(\cF)\le1/4$ in \cite[Theorem~9]{CGFS1986}, despite using an entropy method rather than a spectral method:
If $B$ is a bipartite graph and $F_1, F_2\in\cF$, then $B^c$ must intersect $F_1\cap F_2$ since $B$ cannot contain a triangle.
Concretely put,
\[
    G\in\cF\ \Longrightarrow\ G\mathbin{\triangle}B^c\notin\cF.
\]
Therefore, bipartite graphs encode information about triangle-intersecting families: any bipartite graph $B$ represents a mechanism that converts a graph in $\cF$ to a graph outside of $\cF$ by flipping every edge in $B^c$.

We now apply this observation to our setting. 
$\Far_2(\cA)$ encodes information about $\cA$ in a similar manner:
For every $S\in\Far_2(\cA)$,
\[
    S\cap X^c\cap Y^c\ne\varnothing
    \quad\hbox{for every }X,Y\in\cA.
\]
Consequently, for $S\in\Far_2(\cA)$ and $X\in\cA$, flipping every coordinate of $X$ in $S$ removes it from $\cA$.
Lifting the above construction to the level of operators, we define 
\[
\mathsf T_S f(x)=f(x+S),
\]
where we use bitwise addition (XOR operation) on $\{0,1\}^N$.
An advantage of $\mathsf T_S$ is that whenever $f$ is supported in $\cA$, $\mathsf T_Sf(x) = 0$ for $x\in\cA$.
Therefore, it defines a vanishing quadratic form on functions supported on $\cA$.
This is the building block of the operator constructed in \cite[Section~2.2]{EFF2012} to prove Theorem~\ref{thm:simonovits-sos}.
The operator we eventually use in the proof of Theorems~\ref{thm:quantitative}~and~\ref{thm:quantitative2} resembles the construction in \cite[Section~4.1]{EFF2012}, an extension of $\mathsf T_S$ tailored to arbitrary Bernoulli-$p$ measures.

Once we have an operator whose quadratic form vanishes on functions supported on $\cA$ and whose spectrum has a suitable lower bound, applying the spectral lower bound to the indicator of $\cA$ gives an upper bound on $\mu_p(\cA)$.
This spectral argument is reminiscent of the Hoffman bound, which controls the independence number of a regular graph by the smallest eigenvalue of its adjacency matrix.
We refer the reader to \cite{Haemers2021} for an account of this method.

One potential difficulty of the spectral approach is obtaining good spectral information on the constructed operators.
The good news is that any linear combination of $\mathsf T_S$ for $S\in\Far_2(\cA)$ preserves the required vanishing property.
To solve Problem~\ref{prob:frac-discrete-convexity}, we naturally look to average over the provided spread measure.
This leads to a modification of $\mathsf T_S$ such that the desired lower bound on the spectrum follows from the spread inequalities.
For more sophisticated linear combinations of the operators indexed by bipartite graphs, see \cite[Section~2.4~and~Section~4.2]{EFF2012}.

\subsection{Subsequent developments}
Motivated by finding ``the proper formulation'' of Problem~\ref{prob:frac-discrete-convexity} devoid of unions, Talagrand raised in \cite{Talagrand2010} a sequence of increasingly strong conjectures about $p$-spread measures, each implying a positive solution to the fractional discrete convexity problem.
See Fang and Wang \cite{FW2026} for a proof of \cite[Conjecture~7.8]{Talagrand2010} that adapts the techniques of this paper and establishes that, up to a constant, $p$-spread measures are stochastically dominated by Bernoulli-$p$ measures after independently retaining elements with a fixed probability.
See Park and Talagrand \cite{PT2026} for a proof of a logarithmic projection inequality that resolves \cite[Conjecture~9.1]{Talagrand2010}, providing an alternative route to Theorem~\ref{thm:quantitative} along the same line that exploits spreadness via the Fourier basis but proceeds with a second moment instead of a spectral calculation.

\subsection*{Use of Artificial Intelligence}
The author engaged in many conversations with GPT-5.6 throughout August 2026 to brainstorm proof and (in the early stages) counterexample construction strategies for Problem~\ref{prob:frac-discrete-convexity}.
During this iterative process across numerous chats, the author provided the model with his formulation of the problem in Section~\ref{sec:comparison} in terms of coupling and (its dual) functional inequality, and constantly steered the model away from multiple approaches he deemed unpromising, some of which he had explored earlier.
A GPT-5.6 Sol agent named ``Random coupling'' suggested a proof of Theorem~\ref{thm:quantitative}, calling it a ``Hoffman-type bound''.
The model's chain of thought is unavailable, but this idea resembles those in \cite{Friedgut2008} and \cite{EFF2012}.
The author verified the argument and developed its presentation.

\section{Proof of fractional discrete convexity conjecture}
\label{sec:proof}
We work in the space of functions on $\{0,1\}^N$ with the inner product 
\begin{equation}\label{eq:inner-product}
\langle f,g\rangle_p=\EE_{\mu_p}[fg]=\sum_{I\subseteq[N]}\mu_p(I)f(I)g(I)\qquad \text{for }f,g:2^{[N]}\rightarrow\mathbb{R}.
\end{equation}
The (squared) norm induced by \eqref{eq:inner-product} is
\[
\lVert f\rVert_p^2=\langle f,f\rangle_p=\EE_{\mu_p}f^2.
\]
We use $\one$ to denote the constant-one function (i.e. $\one(I)=1$ for all $I\subseteq[N]$) and $\ind_\mathcal{E}$ to denote the indicator function of $\mathcal{E}\subseteq2^{[N]}$ (i.e. $\ind_\mathcal{E}(I)=\ind_{I\in\mathcal{E}}$).

We now introduce the standard \emph{$p$-skewed Fourier–Walsh basis}.
Define
\[
    \varphi(x)=\frac{x-p}{\sqrt{p(1-p)}}
    \qquad (x\in\{0,1\}).
\]
Taking products, the functions
\begin{equation}\label{eq:basis}
    \chi_I(x)=\prod_{i\in I}\varphi(x_i)\qquad(I\subseteq[N])
\end{equation}
form an orthonormal basis of $L^2(\{0,1\}^N,\mu_p)$.
As a result, we can expand any $f$ on $\{0,1\}^N$ in terms of the basis in \eqref{eq:basis}.
Define the \emph{Fourier coefficients} of $f$ as
\[
\widehat{f}(I)=\langle f,\chi_I\rangle_p\qquad(I\subseteq[N]).
\]
Since $\chi_\emptyset=\one$, the Fourier coefficient at $\emptyset$ is the mean of $f$:
\begin{equation}\label{eq:mean}
\widehat{f}(\emptyset)=\langle f,\one\rangle_p=\EE_{\mu_p}f.
\end{equation}
Therefore, we can write
\begin{equation}\label{eq:fourier}
    f=\EE_{\mu_p}f+g,
\end{equation}
where
\[
g = \sum_{\substack{I\subseteq[N]\\I \neq \emptyset}}\langle f,\chi_I\rangle_p\chi_I.
\]
Note that the variance of $f$ is the squared norm of $g$:
\begin{equation}\label{eq:var}
\operatorname{Var}_{\mu_p}(f)=\EE_{\mu_p}(f-\EE_{\mu_p}f)^2=\EE_{\mu_p}g^2=\lVert g\rVert_p^2.
\end{equation}

The proof proceeds by constructing self-adjoint kernel operators whose quadratic forms vanish on functions supported on $\cA$.
If $\Far_2(\cA)$ carries a $p$-spread measure, the weighted average has a spectral lower bound.
Applying these facts to $\ind_\cA$ forces $\mu_p(\cA)$ to be small.

We first introduce the operator with the desired properties.
We start from the single-coordinate case, adopting the construction in \cite[Section~2.2]{Friedgut2008}.
We try to find the corresponding kernel function $h(x,y)$, where $x,y\in\{0,1\}$.
For two sets in $\cA$, a coordinate of $S\in\Far_2(\cA)$ witnesses that their union does not cover $S$. 
At this coordinate, $x=y=0$.
Since applying the Hoffman bound requires a vanishing condition, we must look for kernels such that $h(0,0)=0$.
We also seek an operator that leaves $\one$ invariant, since the constant component contributes $\mu_p(\cA)^2$ to the final quadratic-form identity.
This leads us to define
\[
    h(x,y)=1+\lambda\varphi(x)\varphi(y),
\]
with 
\[
    \lambda=-\frac{1-p}{p}.
\] 
Explicitly,
\[
    \bigl(h(x,y)\bigr)_{x,y\in\{0,1\}}
    =\begin{pmatrix}
        0 & \dfrac{1}{p}\\
        \dfrac{1}{p} & \dfrac{2p-1}{p^2}
    \end{pmatrix}.
\]
\begin{remark}
Up to normalization, $h$ is uniquely determined by the two key ingredients of a Hoffman-type bound, namely vanishing at $(0,0)$ and diagonalizability in $\{\one,\varphi\}$.
\end{remark}
\begin{remark}
When $p=1/2$, $h$ corresponds to the coordinate flipping example in Section~\ref{sec:intro}.
\end{remark}
Let $\mathsf H$ be the resulting operator in the one-coordinate space:
\[
    (\mathsf Hg)(x)=\EE_{Y\sim\operatorname{Bernoulli}(p)}[h(x,Y)g(Y)].
\]
Then $\mathsf H$ is self-adjoint, and
\[
    \mathsf H\one=\one,
    \qquad
    \mathsf H\varphi=\lambda\varphi.
\]
Thus, its eigenvalues in the orthonormal basis $\{\one,\varphi\}$ are $1$ and $\lambda$.
For $S\subseteq[N]$, define the product kernel:
\begin{equation}\label{eq:kS}
    k_S(x,y) =\prod_{i\in S}h(x_i,y_i).
\end{equation}
\begin{remark}
    This construction is similar to that in \cite[Section~4.1]{EFF2012}, except that coordinates outside $S$ use projection onto constant functions instead of the identity operator.
\end{remark}
Given a probability measure $\nu$ on $2^{[N]}$, define the corresponding self-adjoint operator $\mathsf K_\nu$ by
\[
    (\mathsf K_\nu f)(x)=\EE_{S\sim\nu}\EE_{Y\sim\mu_p}[k_S(x,Y)f(Y)].
\]
When $\nu$ is carried by $\Far_2(\cA)$, $\mathsf K_\nu$ satisfies the desired vanishing condition.
\begin{lemma}\label{lem:vanishing}
$k_S(x,y)=0$ for every $x,y\in\cA$ and $S\in\Far_2(\cA)$.
Consequently, whenever $\nu$ is carried by $\Far_2(\cA)$,
\begin{equation}\label{eq:zero-form}
    \langle \ind_\cA,\mathsf K_\nu \ind_\cA\rangle_p=0.
\end{equation}
\end{lemma}
\begin{proof}
For every $x,y\in\cA$ and $S\in\Far_2(\cA)$, we have $S\nsubseteq x\cup y$.
Choose $i\in S\setminus(x\cup y)$.
Then $x_i=y_i=0$, so the $i$th factor in \eqref{eq:kS} is $h(0,0)=0$.
It follows that $k_S(x,y)=0$.
Summing over $x,y\in\cA$ and $S\in\Far_2(\cA)$ gives \eqref{eq:zero-form}.
\end{proof}

We next diagonalize $\mathsf K_\nu$ and obtain a lower bound on its spectrum when $\nu$ is $p$-spread.
\begin{lemma}\label{lem:lower-spectrum}
For every $I\subseteq[N]$, the function $\chi_I$ is an eigenfunction of $\mathsf K_\nu$ with eigenvalue
\begin{equation}\label{eq:thetaI}
    \theta_I=\lambda^{|I|}\nu(H_I).
\end{equation}
In addition, if $\nu$ is $p$-spread,
\begin{equation}\label{eq:thetaI-bound}
    \theta_I \ge -(1-p).
\end{equation}
\end{lemma}

\begin{proof}
For any function of the form $f(y)=\prod_{i=1}^N f_i(y_i)$,
\begin{align*}
    &\EE_{Y\sim\mu_p}\bigl[k_S(x,Y)f(Y)\bigr]\\
    &\quad=\prod_{i\in S}\EE_{Y_i\sim\operatorname{Bernoulli}(p)}
    \bigl[h(x_i,Y_i)f_i(Y_i)\bigr]
    \prod_{i\notin S}\EE_{Y_i\sim\operatorname{Bernoulli}(p)}
    \bigl[f_i(Y_i)\bigr].
\end{align*}
In other words, for fixed $S$, the operator with kernel $k_S$ is the tensor product of $\mathsf H$ on the coordinates in $S$ and the projection onto the constants on the coordinates outside $S$.
Taking $f_i=\varphi$ on $I$ and $f_i=\one$ outside $I$ gives
\[
 \EE_{Y\sim\mu_p}\bigl[k_S(x,Y)\chi_I(Y)\bigr]
 =\lambda^{|I|}\ind_{\{I\subseteq S\}}\chi_I(x).
\]
Therefore, its eigenvalue on $\chi_I$ is
\[
    \lambda^{|I|}\ind_{\{I\subseteq S\}}.
\]
Averaging over $S\sim\nu$ gives \eqref{eq:thetaI}.

If $|I|$ is even, $\theta_I\ge0$.
If $|I|$ is odd, \eqref{eq:spread} gives
\[
    \theta_I
    =-\left(\frac{1-p}{p}\right)^{|I|}\nu(H_I)
    \ge -\left(\frac{1-p}{p}\right)^{|I|}p^{|I|}
    =-(1-p)^{|I|}
    \ge -(1-p).
\]
\end{proof}

We now combine Lemma~\ref{lem:vanishing} and Lemma~\ref{lem:lower-spectrum} to complete the Hoffman-type bound.
\begin{proof}[Proof of Theorem~\ref{thm:quantitative}]
Let $\nu$ be a $p$-spread measure carried by
$\Far_2(\cA)$, and let
\[
    f=\ind_\cA\qquad\text{and}\qquad a=\widehat{f}(\emptyset).
\]
By \eqref{eq:mean},
\[
a=\EE_{\mu_p}f=\mu_p(\cA)
\]
is the mean of $f$, our quantity of interest.
We expand $f$ on the $p$-skewed Fourier-Walsh basis as in \eqref{eq:fourier}, writing 
\[
f=a+g.
\] 
By \eqref{eq:var},
\[
    \lVert g\rVert_p^2=\operatorname{Var}_{\mu_p}(f)=a(1-a).
\]
According to \eqref{eq:zero-form} and \eqref{eq:thetaI-bound},
\begin{equation}\label{eq:hoffman-bound}
    0
    =\langle f,\mathsf K_\nu f\rangle_p
    =a^2+\langle g,\mathsf K_\nu g\rangle_p
    \ge a^2-(1-p)\lVert g\rVert_p^2 
    =a\bigl((2-p)a-(1-p)\bigr).
\end{equation}
This implies
\[
    a\le\frac{1-p}{2-p}.
\]
\end{proof}
\begin{proof}[Proof of Theorem~\ref{thm:quantitative2}]
By linear-programming duality as in \cite[Proposition~6.5]{Talagrand2010}, failure of weak $p$-smallness is equivalent to $\Far_2(\cA)$ carrying a probability measure $\nu$ such that 
\[
 \nu(H_I)<2p^{|I|}\qquad\text{for every }I\subseteq[N].
\]
As a result, the proof is the same as that for Theorem~\ref{thm:quantitative}, but with a factor of $2$ in the spectral bound \eqref{eq:thetaI-bound}:
\[
 \theta_I>-2(1-p).
\]
With $a=\mu_p(\cA)$, the calculation in \eqref{eq:hoffman-bound} now yields
\[
a<\frac{2-2p}{3-2p}.
\]
\end{proof}

\section{Proof of coupling theorem}\label{sec:comparison}
Theorem~\ref{thm:coup} is analogous to the Gaussian-sum formulation of the convexity problem in \cite[Problem~0.2]{Song2026} and \cite[Problem~2]{HST}.
We prove it by establishing a functional inequality underlying its dual formulation.
For $f\colon 2^{[N]}\to\mathbb R$, define
\[
    Q_2 f(X)=\min_{\substack{X_1,X_2\subseteq[N]\\X\subseteq X_1\cup X_2}}\bigl(f(X_1)+f(X_2)\bigr).
\]

\begin{theorem}\label{thm:fi}
Let $p\in(0,1)$ and $\nu$ be a $p$-spread measure on $2^{[N]}$.
For every $f\colon2^{[N]}\to\mathbb R$,
\begin{equation}\label{eq:fi}
    \EE_\nu Q_2f\le 2\EE_{\mu_p} f.
\end{equation}
\end{theorem}
\begin{proof}
As in the proofs of \cite[Theorem~1.1]{Song2026} and \cite[Proposition~2.6]{Talagrand1995}, our proof relies on the law of large numbers.

We first create some slack in the spread parameter.
Sample $X\sim\nu$ and independently retain each element of $X$ with probability $\rho\in(0,1)$, where $\rho$ is the slack.
Let $\nu_\rho$ be the law of the resulting set.
For every $I\subseteq[N]$,
\begin{equation}\label{eq:thinning-spread}
    \nu_\rho(H_I)
    =\rho^{|I|}\nu(H_I)
    \le(\rho p)^{|I|}.
\end{equation}
In other words, $\nu_\rho$ is $(\rho p)$-spread.
By bounded convergence,
\[
    \EE_{\nu_\rho}Q_2f\longrightarrow \EE_\nu Q_2f
    \qquad\text{as }\rho\uparrow1.
\]

We now take tensor powers. 
On $[N]\times[L]$, write
\[
    S^{(\ell)}=\{i\in[N]:(i,\ell)\in S\}
\]
for the restriction of $S\subseteq[N]\times[L]$ to block $\ell$.
Define
\[
    f_L(S)=\frac{1}{L}\sum_{\ell=1}^L f(S^{(\ell)}).
\]
Note that because of the product structure,
\[
    Q_2f_L(S)=\frac{1}{L}\sum_{\ell=1}^L Q_2f(S^{(\ell)}).
\]
Fix $\varepsilon>0$.
Consider
\begin{align}
    \cA_L&=\{R:f_L(R)\le\EE_{\mu_p}f+\varepsilon\},\label{eq:tensor-family}\\
    \cB_L&=\{S:Q_2f_L(S)\ge\EE_{\nu_\rho}Q_2f-\varepsilon\}.\label{eq:tensor-bad-event}
\end{align}
By the law of large numbers, we can choose $L$ large enough such that
\begin{align}
    &\mu_p^{\otimes L}(\cA_L)>\frac12,\label{eq:large-A}\\
    &\nu_\rho^{\otimes L}(\cB_L)\ge\rho,\label{eq:large-B}
\end{align}
where $\mu_p^{\otimes L}$ and $\nu_\rho^{\otimes L}$ denote the product measures of $\mu_p$ and $\nu_\rho$ on $(2^{[N]})^L\cong2^{[N]\times[L]}$, respectively.

Finally, define $\nu_L=\nu_\rho^{\otimes L}(\,\cdot\mid\cB_L)$.
In particular, $\nu_L$ is carried by $\cB_L$.
We claim that it is also $p$-spread.
Indeed, for every nonempty $I\subseteq[N]\times[L]$, by
\eqref{eq:thinning-spread} and \eqref{eq:large-B},
\begin{equation}\label{eq:p-spread}
    \nu_L(H_I)
    \le\frac{\nu_\rho^{\otimes L}(H_I)}{\nu_\rho^{\otimes L}(\cB_L)}
    \le\frac{(\rho p)^{|I|}}{\rho}
    \le p^{|I|}.
\end{equation}
In addition, the inequality holds trivially for $I=\varnothing$.
Therefore, in light of \eqref{eq:large-A} and \eqref{eq:p-spread}, Corollary~\ref{cor:frac} implies that there exist sets $A_1,A_2\in\cA_L$ and $B\in\cB_L$ such that $B\subseteq A_1\cup A_2$.
Putting together \eqref{eq:tensor-family} and \eqref{eq:tensor-bad-event}, we have
\[
    \EE_{\nu_\rho}Q_2f-\varepsilon\le Q_2f_L(B)\le f_L(A_1)+f_L(A_2)\le2\EE_{\mu_p}f+2\varepsilon.
\]
Letting $\varepsilon\to0$ and then $\rho\to1$ gives \eqref{eq:fi}.
\end{proof}

Theorem~\ref{thm:coup} follows from Theorem~\ref{thm:fi}.
As in Talagrand's proof of \cite[Proposition~5.1]{HST} and \cite[Lemma~2.6]{Talagrand2026}, our proof mechanism is a version of the Hahn-Banach theorem known as Farkas' lemma.
\begin{proof}[Proof of Theorem~\ref{thm:coup}]
Define
\[
    \Gamma=\bigl\{(X,R_1,R_2)\in(2^{[N]})^3:X\subseteq R_1\cup R_2\bigr\}.
\]
The desired coupling exists precisely when we can choose a nonnegative function $\pi(X,R_1,R_2)$ on $\Gamma$ such that
\begin{align}
    \sum_{\substack{R_1,R_2:\\(X,R_1,R_2)\in\Gamma}}
    \pi(X,R_1,R_2)
    &=\nu(X) &&(X\in2^{[N]}),\label{eq:x-margin}\\
    \sum_{\substack{(X,R_1,R_2)\in\Gamma:\\R_j=R}}
    \pi(X,R_1,R_2)
    &=\mu_p(R) &&(j\in\{1,2\},\ R\in2^{[N]}).\label{eq:r-margin}
\end{align}
We can write the system of equations \eqref{eq:x-margin} and \eqref{eq:r-margin} in matrix form:
\begin{equation}\label{eq:system}
    M\pi=b,\qquad \pi\geq0,
\end{equation}
where $M$ is a zero-one matrix and $b$ is a concatenation of $\nu$ and two copies of $\mu_p$.
By Farkas' lemma, \eqref{eq:system} is solvable if and only if
\[
    M^{\mathsf T}u\geq0\Longrightarrow b^{\mathsf T}u\geq0\quad\text{for every $u$}.
\]
Decompose $u$ into three functions: $-h(X)$ for the coordinates in \eqref{eq:x-margin}; $g_1(R_1)$ and $g_2(R_2)$ for the coordinates in \eqref{eq:r-margin}. 
The condition $M^{\mathsf T}u\geq0$ amounts to
\[
    h(X)\leq g_1(R_1)+g_2(R_2)\quad\text{whenever }X\subseteq R_1\cup R_2.
\]
Let $g=(g_1+g_2)/2$.  
Applying the preceding inequality both to $(R_1,R_2)$ and to $(R_2,R_1)$, and then averaging, gives
\[
    h(X)\le g(R_1)+g(R_2)
    \qquad\text{whenever }X\subseteq R_1\cup R_2.
\]
Hence $h\leq Q_2g$.
By Theorem~\ref{thm:fi}, we have 
\[
    \EE_\nu h
    \leq \EE_\nu Q_2g
    \leq 2\EE_{\mu_p}g
    = \EE_{\mu_p}g_1+\EE_{\mu_p}g_2,
\]
which is exactly $b^{\mathsf T}u\geq0$.
\end{proof}

\begin{remark}
Note that Theorem~\ref{thm:coup} recovers Corollary~\ref{cor:frac}.  
Indeed, if $\mu_p(\cA)>1/2$, the union bound gives
\[
     \PP(R_1,R_2\in\cA)
     \geq1-\PP(R_1\notin\cA)-\PP(R_2\notin\cA)
     \geq2\mu_p(\cA)-1>0.
\]
On this event, $X\in\Cov_2(\cA)$, so $\nu$ puts positive probability on $\Cov_2(\cA)$.
\end{remark}

\section*{Acknowledgments}
The author thanks Jinyoung Park for piquing his interest in this area and for feedback on an earlier version of the manuscript.

\bibliographystyle{amsplain}

\end{document}